\documentclass[a4paper,12pt]{amsart}

 \usepackage{amssymb,amsthm}
\usepackage{amscd} 
\usepackage{graphicx,color} 

\usepackage{url} 

\newtheorem{theorem}{Theorem}[section]
\newtheorem{corollary}[theorem]{Corollary}

\theoremstyle{definition}

\newtheorem{remark}[theorem]{Remark}

\def\r{\mathbb R}
\def\c{\mathbb C}
\def\n{\mathbb N}
\def\z{\mathbb Z}
\def\s{\mathbb S}
\newcommand{\T}{\hbox{\bf t}}
\newcommand{\N}{\hbox{\bf n}}
\newcommand{\B}{\hbox{\bf b}}
\newcommand{\C}{\hbox{\bf c}}

\begin{document}
\title{Two classification results for  stationary surfaces  of the least moment of inertia}


\author{Rafael L\'opez}






\begin{abstract}
	A surface in  Euclidean space $\r^3$ is said to be an $\alpha$-stationary surface if it is a critical point of the energy $\int_\Sigma|p|^\alpha$, where $\alpha\in\r$.   We prove that all ruled $\alpha$-stationary surfaces are vector planes (for all $\alpha$) and a type of elongated helicoids (for $\alpha=1$).    The second result of classification asserts that if $\alpha\not=-2,-4$, any $\alpha$-stationary surface  foliated by circles must be a rotational surface.  If $\alpha=-4$, the surface is the inversion of a plane, a helicoid, a catenoid or an Riemann minimal example.  If $\alpha=-2$,  we find many non-spherical cyclic $(-2)$-stationary surfaces.
\end{abstract}

\maketitle

\section{Introduction and statement of the results} 

In this paper we investigate surfaces $\Sigma$ in Euclidean space $\r^3$ which are critical points of the energy 
$$E_\alpha[\Sigma]=\int_\Sigma|p|^\alpha\, d\Sigma,$$
where   $\alpha\in\r$ is a parameter and $|p|$  denotes the modulus of a generic point $p\in\Sigma$. Critical points of this energy are called  {\it $\alpha$-stationary surfaces} and they    are characterized by the Euler-Lagrange equation
\begin{equation}\label{eq1}
H(p)=\alpha\frac{\langle N(p),p\rangle}{|p|^2},\quad p\in\Sigma,
\end{equation}
where $H$ and $N$ are the mean curvature and the unit normal vector of $\Sigma$, respectively. We are assuming   that the origin $0\in\r^3$ does not belong to $\Sigma$. In \eqref{eq1}, the mean curvature $H$ is the sum of the principal curvatures of the surface. If $\alpha=0$, then $E_0[\Sigma]$ is simply the area of $\Sigma$ and $0$-stationary surfaces are minimal surfaces. From now, we will discard the case $\alpha=0$. Another interesting case is  $\alpha=2$ because the energy $E_{2}$ represents the moment of inertia of $\Sigma$ with respect to the origin of $\r^3$. 

Stationary surfaces have been recently studied by Dierkes and Huisken in \cite{d4,dh} as an extension  in arbitrary dimensions of an old problem investigated by Euler on planar curves. In Euclidean plane $\r^2$, Euler asked for those planar curves  $\gamma$   of constant density joining two given points of $\r^2$ with least energy $\int_\gamma |\gamma|^\alpha\, ds$: see    \cite[p. 53]{eu}.   Euler found explicit parametrizations  of these curves in polar coordinates (see also \cite{ca,to}). Later,  Mason found all minimizers in the particular case $\alpha=2$ \cite{ma}.  Recently, the author has obtained a variety of results for $\alpha$-stationary surfaces,   also in collaboration with Dierkes. For example, existence of the Plateau problem \cite{dl1}, characterizations of the rotational surfaces \cite{dl2}, a relationship between $\alpha$-stationary surfaces and   minimal surfaces \cite{lo1} or characterizing those stationary surfaces with constant Gauss curvature \cite{lo2}. 

Some examples of $\alpha$-stationary surfaces are the following (see Fig. \ref{fig0}):
\begin{enumerate}
\item Vector planes. A plane of $\r^3$ is an $\alpha$-stationary surface if and only if the plane  contains $0$, that is, it is a vector plane. This holds for any $\alpha\in\r$.  
\item Spheres. The only spheres that are $\alpha$-stationary surfaces are spheres centered at $0$ ($\alpha=-2$) and spheres containing $0$ ($\alpha=-4$).
\end{enumerate}
 
  \begin{figure}[hbtp]
\begin{center}
\includegraphics[width=.3\textwidth]{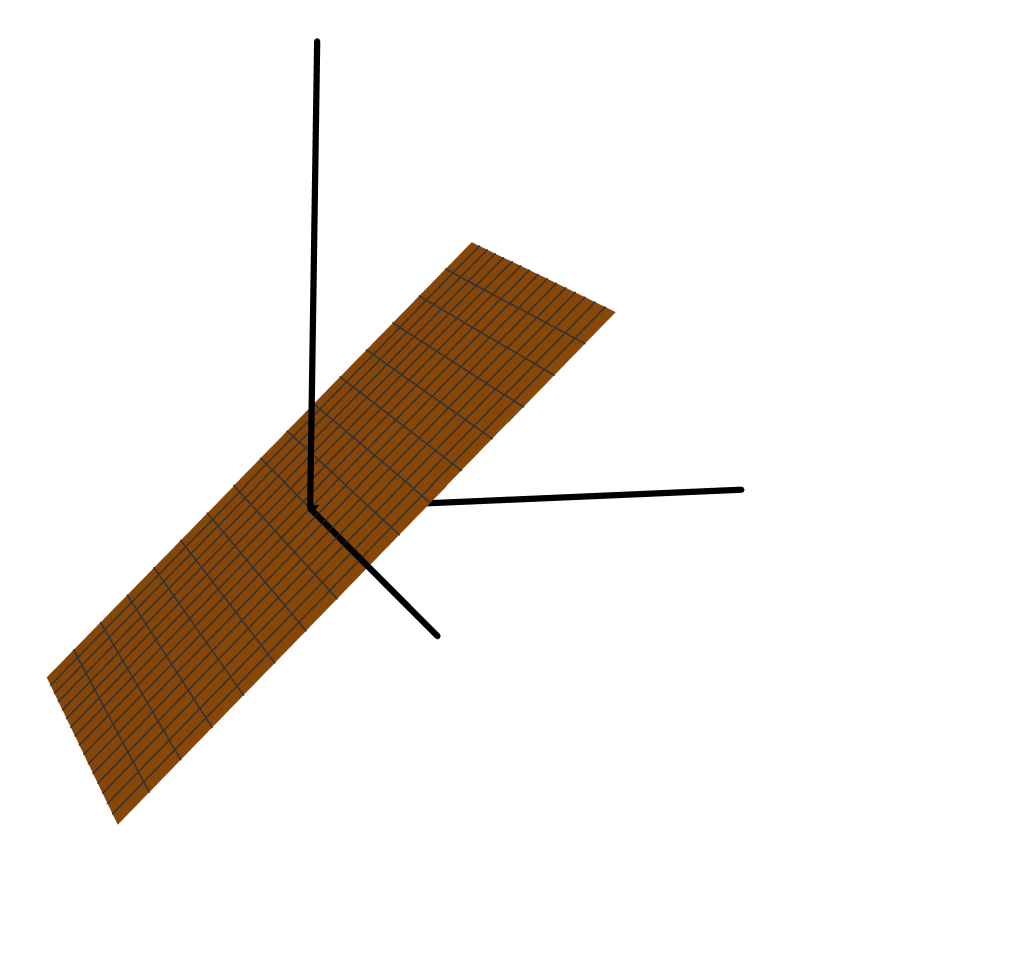},\quad \includegraphics[width=.25\textwidth]{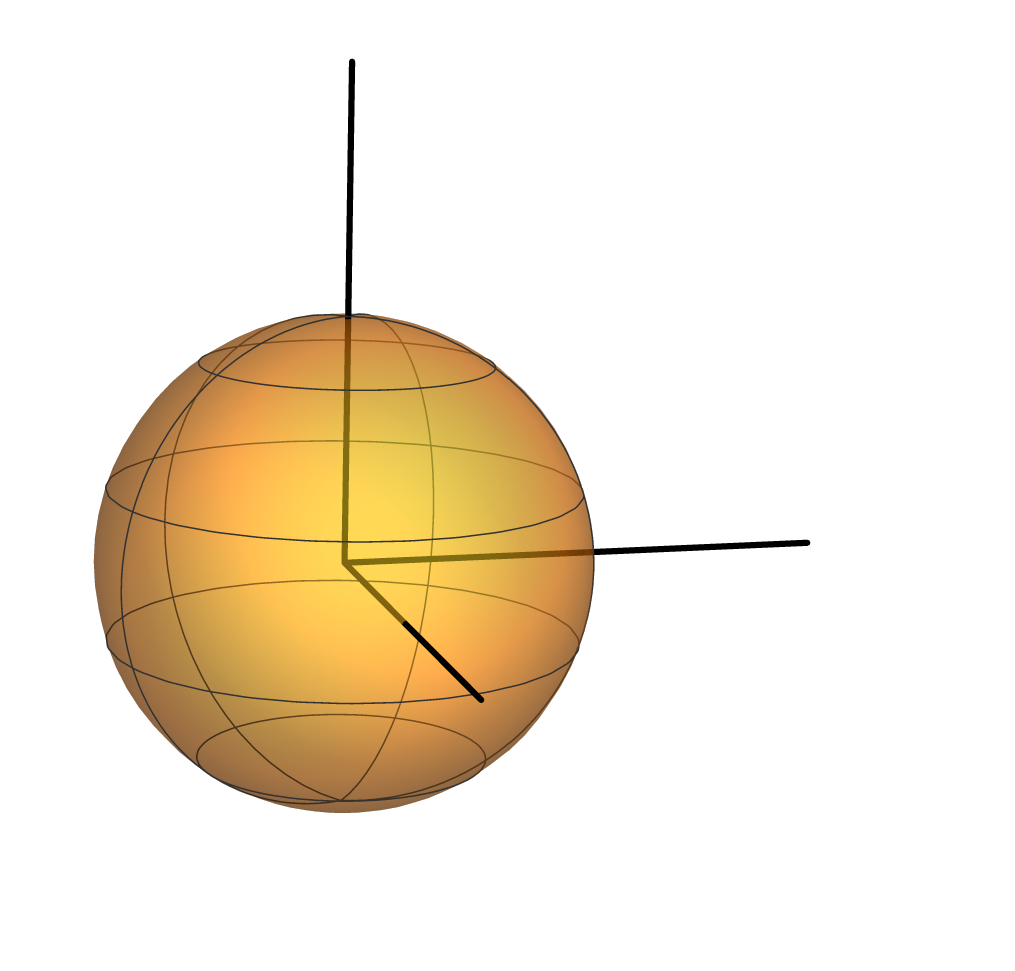}\quad \includegraphics[width=.2\textwidth]{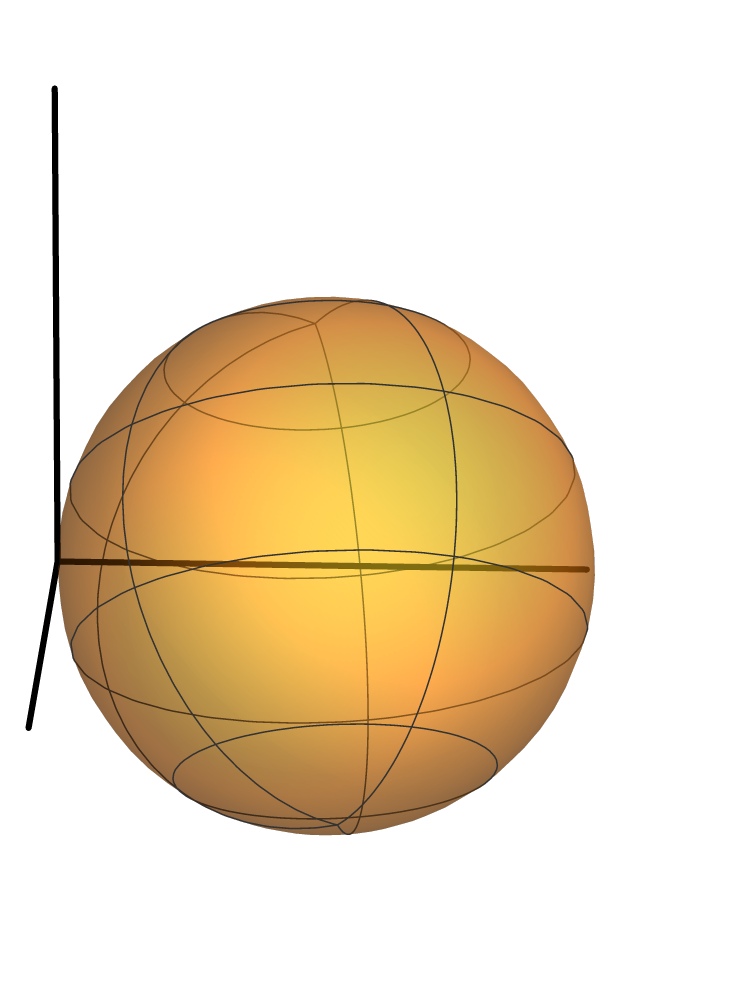}
\end{center}
\caption{Examples of $\alpha$-stationary surfaces: vector planes (left), spheres centered at $0$ (middle) and spheres crossing $0$ (right).}\label{fig0}
\end{figure}
We point out that the energy $E_\alpha$ is not preserved by rigid motions of $\r^3$ and only rigid motions that fix $0$, that is, linear isometries of $\r^3$, preserve the energy $E_\alpha$.  For example, if $\Sigma$ is an $\alpha$-stationary surface, the translated surface $\Sigma+\vec{v}$, $\vec{v}\in\r^3$, is not an $\alpha$-stationary surface.   Dilations from $0$ carry $\alpha$-stationary surfaces in $\alpha$-stationary surfaces. 

It is natural to classify $\alpha$-stationary surfaces assuming some type of geometric property. Recall that  $\alpha$-stationary surfaces of rotational type were investigated and classified in \cite{dl2}. In this paper, we give two types of classification. First, we find all $\alpha$-stationary surfaces which also are ruled surfaces. 

\begin{theorem}\label{t1}
The only ruled $\alpha$-stationary surfaces are:
\begin{enumerate}
\item vector planes, for all $\alpha\in\r$, and
\item $\alpha=1$ and the    surfaces are parametrized by
\begin{equation}\label{heli}
\Psi(s,t)=(0,0,me^s)+t(\cos s,\sin s,0),\quad s,t,m\in\r.\end{equation}
\end{enumerate}
\end{theorem}
Notice that the surface \eqref{heli} looks like the standard helicoid $(s,t)\mapsto (0,0,ms)+t(\cos s,\sin s,0)$ with the difference that the parameter $s$ of the basis curve is replaced by $e^s$. This makes   the third coordinate of the helicoid to elongate exponentially as $s\to\infty$ and, in contrast, when $s$ tends to $-\infty$,  the helicoid flattens out on the plane of equation $z=0$.   

It is not expectable that a cylindrical surface constructed with base an $\alpha$-stationary planar curve may be an $\alpha$-stationary surface. To be precise, let $\gamma$ be   an $\alpha$-stationary  in $\r^2$.  Then the cylindrical surface parametrized by $(s,t)\mapsto (\gamma(s),t)$ is not in general an $\alpha$-stationary surface. This is because for these cylindrical surfaces, all terms of Eq. \eqref{eq1} are preserved  except the denominator in the second hand-side of this equation. See Rem. \ref{r21} below. In fact, Theorem \ref{t1} implies that the only $\alpha$-stationary surfaces of cylindrical type are vector planes (see also Thm. \ref{t21}).

The second type of  results  concerns to the classification of $\alpha$-stationary surfaces of cyclic type. A {\it cyclic surface} is surface   defined as a smooth one-parameter family of   circles, whose radii may change along the parameter. Equivalently, a cyclic surface is a surface such that there is a smooth one-parameter of planes that intersect the surface in circles. We also say that the surface is foliated by circles. An example of a cyclic surface is a rotational surface, where all the circles of the foliation are coaxial with respect to the axis of the surface.   On the other hand,   spheres are rotational surfaces but it can be also viewed as cyclic surfaces when we intersect a sphere for any smooth one-parameter family of planes. In such a case, a sphere  can be foliated by different types of families of circles.    A first result considers the case that   the planes of the foliation are parallel.

\begin{theorem}\label{t2}
Let $\Sigma$ be a cyclic $\alpha$-stationary surface.  If the planes of the foliation are parallel, then   $\Sigma$ is a rotational surface.
\end{theorem}

If the planes of the foliation are not parallel, we distinguish cases according to $\alpha$.

\begin{theorem}\label{t3} Let $\Sigma$ be a cyclic $\alpha$-stationary surface. If the planes of the foliation are not parallel, then $\alpha\in \{-5,-4,-2\}$. Furthermore, 
\begin{enumerate} 
\item If $\alpha=-5$, then $\Sigma$ is the inversion with respect to the origin of the ruled surface \eqref{heli}.
\item If $\alpha=-4$, then  $\Sigma$ is an       inversion with respect to the origin  of a plane, a helicoid, a catenoid or a Riemann minimal example.
\end{enumerate}
\end{theorem}

As a consequence of Thms. \ref{t2} and \ref{t3}, we have a classification of $\alpha$-stationary surfaces when $\alpha\not= -2,-4$.

\begin{corollary} Any cyclic $\alpha$-stationary surface with $\alpha\not\in \{-5,-4,-2\}$ is a rotational surface.
\end{corollary}

If $\alpha=-2$, we will prove that there are many examples of non-spherical cyclic  $(-2)$-stationary surfaces.  See Fig. \ref{fig2} in Sect. \ref{s7}.

The paper is organized as follows. Theorem \ref{t1} will be proved in   Sects. \ref{s2} and \ref{s3}: we will distinguish if the surface is or is not of cylindrical type. In Sect. \ref{s4} we prove Thm. \ref{t2}. In Sect. \ref{s5} we prove Thm. \ref{t3} when $\alpha=-4$. Section \ref{s6} gives the proof of Thm. \ref{t3} when $\alpha\not=-2,-4$ and finally, Section \ref{s7} is devoted to the case   $\alpha=-2$. 
\section{Proof of Theorem \ref{t1}: part I}\label{s2}

 The proof of Thm. \ref{t1} is done in two steps. First, the theorem is proved when the ruled surface is a cylindrical surface. In this case, we will prove in this section that   the surface is a vector plane.   In  Sect. \ref{s3} we will consider the case of ruled non-cylindrical surfaces. In such a case, we prove that it is a vector plane or it is the surface parametrized by  \eqref{heli}.
 
A ruled surface is a surface generated by moving a straight line in Euclidean space $\r^3$. A ruled surface can be parametrized by 
\begin{equation}\label{para1}
\Psi\colon I\times\r\to\r^3,\quad \Psi(s,t)=\gamma(s)+t\beta(s),
\end{equation}
where $\gamma \colon I\to\r^3$ is a regular curve, called the  directrix  which we can suppose to be parametrized by arc-length. The curve $\beta$ is a curve on the unit $2$-sphere $\beta\colon I\to\s^2$ and it indicates the direction of the  rulings of the surface  at each point $\alpha(s)$. If $\beta$ is a constant curve, $\beta(s)=\vec{w}$, then the surface is called {\it cylindrical}. In the following result, we classify all cylindrical $\alpha$-stationary  surfaces.   

\begin{theorem} \label{t21}
Vector planes are the   only cylindrical $\alpha$-stationary surfaces.
\end{theorem}

In consequence, this result is included in   the item (1) of Thm. \ref{t1}.
\begin{proof} 
Let $\Sigma$ be a cylindrical surface. A particular case of cylindrical surface is when $\gamma$ is contained in a plane and $\vec{w}$ is parallel to this plane. Then $\Sigma$ is (part of) a plane and, in such a case, if the surface is stationary, then $\Sigma$ is a vector plane. This proves the result. 

Suppose now the general case. Then   the curve $\gamma$ can be replaced by another curve contained in a plane perpendicular to $\vec{w}$. In particular, $\gamma$ is not a straight-line.   Then the parametrization \eqref{para1} can be written as
   \begin{equation}\label{para2}
\Psi\colon I\times\r\to\r^3,\quad \Psi(s,t)=\gamma(s)+t\vec{w},
\end{equation}
where $\gamma \colon I\to\r^3$ is a   curve  parametrized by arc-length, $|\vec{w}|=1$. The mean curvature $H$ is   the curvature  $\kappa$ of $\gamma$.   Let $\{\gamma',{\bf n}, {\bf b}\}$ the Frenet frame. Since $\gamma$ is a planar curve, then ${\bf b}$ is constant so, ${\bf b}=\pm \vec{w}$. After a change of $\vec{w}$ by $-\vec{w}$ if necessary, we can assume $\vec{w}=-{\bf b}$. In consequence, the unit normal of $\Sigma$ is $N=\Psi_s\times\Psi_t=\gamma'\times \vec{w}={\bf n}$. Then Eq. \eqref{eq1} becomes
$$\kappa(s)=\alpha\frac{\langle{\bf n}(s),\Psi(s,t)\rangle}{|\Psi(s,t)|^2}=\alpha\frac{\langle{\bf n}(s),\gamma(s)+t\vec{w}\rangle}{|\gamma(s)|^2+t^2}=\alpha\frac{\langle{\bf n}(s),\gamma(s) \rangle}{|\gamma(s)|^2+t^2}.$$
This identity can be expressed as the  polynomial equation on $t$
$$\kappa(s)t^2+\kappa(s)|\gamma(s)|^2-\alpha\langle{\bf n}(s),\gamma(s)\rangle=0.$$
Then the coefficient of $t^2$ implies $\kappa(s)=0$ identically,  which it is a contradiction.
\end{proof}

\begin{remark} \label{r21}As we said in the Introduction,  the cylindrical surface constructed with an  $\alpha$-stationary planar curve   is not an $\alpha$-stationary surface unless that the surface is a vector plane.    This situation differs from other similar  surfaces whose  mean curvature satisfies a   PDE similar to \eqref{eq1}, and where cylindrical surfaces are   not trivial. In order to show these examples, consider $(x,y,z)$ canonical coordinates of $\r^3$ and let  $N=(N_1,N_2,N_3)$.   
\begin{enumerate}
\item Translating solitons. These surfaces are critical points of the energy $\int_\Sigma e^z\, d\Sigma$. These surfaces are characterized by $H=N_3$.
\item Self-shrinkers. These surfaces are critical points of the energy $\int_\Sigma e^{-|p|^2}\, d\Sigma$. These surfaces are characterized by $H=-\langle N,p\rangle$.
\item Singular minimal surfaces. These surfaces are critical points of the energy $\int_\Sigma z^\alpha\, d\Sigma$. These surfaces are characterized by  $H=\alpha N_3/z$.  
\end{enumerate}
In all these cases, the planar curves $\gamma$ that are critical points of the one-dimensional case provide critical points  in the two-dimensional case by   considering the corresponding cylindrical surfaces $(s,t)\mapsto (\gamma(s),t)$: \cite{al,lo01,lo02}. On the other hand, also in these three cases, any ruled surface  that it is a critical point  of the energy must be cylindrical surface: \cite{an,ay,hi,lo03}. 

\end{remark}

\section{Proof of Theorem \ref{t1}: part II}\label{s3}

The second step in the proof of Thm. \ref{t1} consist in to prove that if the surface is not cylindrical, then it is a vector planer or $\alpha=1$ and the surface is parametrized by \eqref{heli}.  

 Let $\Sigma$ be a ruled non-cylindrical $\alpha$-stationary surface. By using the parametrization \eqref{para1},   we have that  $\{\beta,\beta',\beta\times\beta'\}$ is an orthonormal frame. We can assume that $\gamma$ is the striction line of the surface. This implies that $\gamma'$ is orthogonal to $\beta'$, $\langle \gamma'(s),\beta'(s)\rangle=0$, for all $s\in I$. Notice that $\gamma$ is not necessarily parametrized by arc-length. Let $\{E,F,G\}$ be the coefficients of the first fundamental form with respect to $\Psi$. The unit normal vector of $\Sigma$ is   
\begin{equation}\label{n1}
N=\frac{\Psi_s\times\Psi_t}{\sqrt{EG-F^2}}=\frac{(\gamma'+t\beta')\times \beta}{\sqrt{EG-F^2}}.
\end{equation}
 Since $\Psi_{tt}=0$, the mean curvature $H$ is calculated by the formula
\begin{equation}\label{n2}
H=\frac{G(\Psi_s,\Psi_t,\Psi_{ss})-2F(\Psi_s,\Psi_t,\Psi_{st})}{(EG-F^2)^{3/2}}.
\end{equation}
 where  $(u,v,w)$ denotes the determinant of three vectors $u,v,w\in\r^3$. Using \eqref{n1} and \eqref{n2}, together with   \eqref{eq1}, we have
 $$|\Psi|^2(G(\Psi_s,\Psi_t,\Psi_{ss})-2F(\Psi_s,\Psi_t,\Psi_{st}))-\alpha \langle(\gamma'+t\beta')\times \beta,\Psi\rangle(EG-F^2)=0.$$
 Computing all terms of this equation, we obtain 
\begin{equation*}
\begin{split}
0&= (t^2+2t\langle\gamma,\beta\rangle+|\gamma|^2)\left((\gamma'+t\beta',\beta,\gamma''+t\beta'')-2\langle\gamma',\beta\rangle(\gamma'+t\beta',\beta,\beta')\right)\\
&-\alpha(\gamma'+t\beta',\beta,\gamma+t\beta)(|\gamma'|^2+t^2|\beta'|^2-\langle \gamma',\beta\rangle^2).
\end{split}
\end{equation*}
By simplifying, we arrive to
$$(t^2+2t\langle\gamma,\beta\rangle+|\gamma|^2)\left((\gamma'+t\beta',\beta,\gamma''+t\beta'')-2\langle\gamma',\beta\rangle(\gamma',\beta,\beta')\right)-\alpha(\gamma'+t\beta',\beta,\gamma)(|\gamma'|^2+t^2|\beta'|^2-\langle \gamma',\beta\rangle^2)=0.$$ 
 This is a polynomial equation on $t$ of degree $4$, 
 \begin{equation}\label{pol}
 \sum_{n=0}^4A_n(s)t^n=0.
 \end{equation}
 Thus all coefficients must vanish identically in $I$. The coefficient $A_4$ is 
 $$A_4=(\beta',\beta,\beta'').$$  
Therefore, $A_4=0$ gives $(\beta',\beta,\beta'')=0$. This is equivalent to say  that $\beta$ is a geodesic of $\s^2$. After a linear isometry of $\r^3$, which does not affect to Eq. \eqref{eq1} neither to the condition that the surface is ruled, we can assume that $\beta$ is the horizontal equator of $\s^2$.   A reparametrization of $\beta$ does not change the condition that $\gamma$ is  a striction line, that is, the   orthogonality between $\beta'$ and  $\gamma'$. From now, we assume 
\begin{equation}\label{b5}
\beta(s)=(\cos s,\sin s,0).
\end{equation}
 Moreover, we have $\beta\times\beta'= e_3:=(0,0,1)$ and $\beta''=-\beta$. 

The computation of $A_3$ gives
$$A_3=(\beta',\beta,\gamma'')+(\gamma',\beta,\beta'')-\alpha(\beta',\beta,\gamma).$$
Using $(\gamma',\beta,\beta'')=-\langle\gamma',\beta'\rangle=0$,   the equation $A_3=0$ is
\begin{equation}\label{g2}
\langle\gamma''(s)-\alpha\gamma(s),e_3\rangle=0
\end{equation}
for all $s\in I$. The computation of the rest of coefficients of the polynomial equation \eqref{pol} gives
\begin{equation}\label{as}
\begin{split}
A_2=&-\alpha(\gamma',\beta,\gamma)+(\gamma',\beta,\gamma'')+2\langle\gamma,\beta\rangle(\beta',\beta,\gamma'')-2\langle\gamma',\beta\rangle(\gamma',\beta,\beta'),\\
A_1=&|\gamma|^2(\beta',\beta,\gamma'')+2\langle\gamma,\beta\rangle\left((\gamma',\beta,\gamma'')-2\langle\gamma',\beta\rangle(\gamma',\beta,\beta')\right)-\alpha(\beta',\beta,\gamma)(|\gamma'|^2-\langle\gamma',\beta\rangle^2),\\
A_0=&|\gamma|^2\left((\gamma',\beta,\gamma'')-2\langle\gamma',\beta\rangle(\gamma',\beta,\beta')\right)-\alpha(|\gamma'|^2-\langle\gamma',\beta\rangle^2)(\gamma,\gamma',\beta).
\end{split}
\end{equation}
With respect to  the orthonormal basis $\{\beta,\beta',e_3\}$,  we express $\gamma$ in coordinates with respect to  this basis, 
$\gamma(s)=(a(s),b(s),c(s))$, where $a,b,c$ are smooth functions on $I$. We also write in coordinates $\gamma'$ and $\gamma''$ with respect to the same basis, using that $\beta''=-\beta$ and $e_3'=0$. Then we obtain  
\begin{equation}\label{abc}
\begin{split}
\gamma&=(a,b,c),\\
\gamma'&=(a'-b,a+b',c'),\\
\gamma''&=(a''-a-2b',2a'+b''-b,c'').
\end{split}
\end{equation}
Since $\gamma$ is the striction line of $\Sigma$,   the condition $\langle\gamma',\beta'\rangle=0$  is equivalent to $a+b'=0$, that is,  $a=-b'$. On the other hand, equation \eqref{g2} writes now as 
  \begin{equation}\label{g22}
  c''=\alpha c.
  \end{equation}
   With this information,  we write again \eqref{abc}, obtaining
\begin{equation*}
\begin{split}
\gamma&=(-b',b,c),\\
\gamma'&=(-b''-b,0,c'),\\
\gamma''&=(-b'''-b',-b''-b,\alpha c).
\end{split}
\end{equation*}
 We impose $A_2=A_1=A_0=0$ in equations \eqref{as} and we write these equations in terms of $a$, $b$ and $c$, obtaining respectively
\begin{align}
0&=c'(b'' +(1-\alpha) b)+2 \alpha  c b',\label{a1}\\
0&=\alpha  c \left(|\gamma|^2- c'^2\right)+2 b' c' (b''+b),\label{a2}\\
0&=c' \left( |\gamma|^2(b''+b)-\alpha  bc'^2 \right)\label{a3}
\end{align}
From \eqref{a3}, we discuss two cases.
\begin{enumerate}
\item Case $ c'=0$ identically.  From \eqref{g22}, we obtain  $c=0$. This   gives $\gamma(s)=0$ and, consequently, $ \Sigma$ is the plane of equation   $z=0$. This case is included in the item (1) of Thm. \ref{t1}.
\item Case $c'\not=0$ at some  $s_0\in I$.  Then $c'\not=0$ in a subinterval $\widetilde{I}\subset I$ around $s_0$. From now on, we work in $\widetilde{I}$. Using   \eqref{a3}, we want to get $|\gamma|^2$. We distinguish two cases.
\begin{enumerate}
\item Case $b''+b=0$. Then \eqref{a3} gives $b=0$. Equation \eqref{a2} is now $|\gamma|^2=c'^2$. This implies $c'^2=c^2$ which, together with \eqref{g22} yields $\alpha=1$ and $c(s)=me^s$ for some $m\in\r$. From the first equation of \eqref{abc}, we have $\gamma(s)=c(s)e_3=(0,0,me^s)$. Since $\beta$ is given in \eqref{b5}, we have proved  the item (2) of Thm. \ref{t1}.
\item Case $b''+b'\not=0$. Then \eqref{a3} implies
$$|\gamma|^2=\alpha\frac{bc'^2}{b''+b}.$$
Replacing in \eqref{a2}, we have 
\begin{equation}\label{a4}
\alpha cc'(\alpha b-(b''+b))+2b'(b''+b)^2=0.
\end{equation}
From \eqref{a1}, we know 
$$b''+b=\frac{\alpha bc'-2\alpha b'c}{c'},$$
which putting in \eqref{a4} gives
$$b'\left(\alpha^2c^2+(b''+b)^2\right)=0.$$
The parenthesis cannot be $0$, so $b'=0$. This implies that $b(s)=b$ is a non-zero constant function because $b''+b\not=0$. Coming back to \eqref{a1}, we have $(1-\alpha)b=0$. Thus $\alpha=1$. Now \eqref{a3} yields $|\gamma|^2=c'^2$, that is, $b^2+c^2=c'^2$. This is in contradiction with \eqref{g22}. This completes the proof of the result.

\end{enumerate} 
 \end{enumerate}

\section{Proof of Theorem \ref{t2}}\label{s4}

Let $\Sigma$ be a cyclic $\alpha$-stationary surface where  the planes containing the circles of the foliation are parallel.  After a linear isometry, we can assume that these planes are parallel to the $xy$-plane. Thus a  parametrization of $\Sigma$ is 
$$\Psi(u,v)=(a(u),b(u),u)+r(u)(\cos{v},\sin{v},0),\quad (u,v)\in I\times\r,$$
where  $a, b$ and $r$ are smooth function in some interval $I\subset\r$
and  $r=r(u)>0$ denotes the radius of each circle of the foliation. The surface $\Sigma$ is rotational if and only if the functions $a$ and $b$ are constant. The computation of  Eq. \eqref{eq1} gives an equation  of type
\begin{equation*}
\sum_{n=0}^{3} A_n(u) \cos{(n v)}+B_n(u)\sin{(n v)}=0.
\end{equation*}
Since the trigonometric functions $\cos(nv)$ and $\sin(nv)$ are linearly independent,  the functions $A_n$ and $B_n$ vanish identically on $I$.  For $n=3$, we obtain
\begin{equation*}
\begin{split}
A_3&=\frac{1}{4} \alpha  r^3 \left(-2 b a' b'-\left(a-3 u a'\right) b'^2+\left(a-u a'\right) a'^2\right)\\
B_3&=\frac{1}{4} \alpha  r^3 \left(a' \left(2 a-3 u a'\right) b'+b \left(a'^2-b'^2\right)+u b'^3\right).
\end{split}
\end{equation*}
The linear combination
$$(3a'b'^2-a'^3)A_3+(3a'^2b'-b'^3)B_3=0$$
implies
$$a'^6-15 a'^4 b'^2+15 a'^2 b'^4-b'^6=0.$$
This equation can be expressed as
\begin{equation}\label{c1}
(a'^2-b'^2)\left((a'^2-b'^2)^2-12 a'^2b'^2\right)=0.
\end{equation}
We have two possibilities: either $a'^2-b'^2=0$ identically in $I$ or $a'^2-b'^2\not=0$ at some point $u=u_0\in I$. In the second case, 
 in a subinterval of $I$ around $u_0$, we have $a'^2-b'^2\not=0$. Thus  \eqref{c1} is equivalent to $(a'^2-b'^2)^2-12 a'^2b'^2=0$. This gives $b'^2=(7\pm 4\sqrt{3})a'^2$.  As a conclusion, the cases $a'^2-b'^2=0$ identically or $a'^2-b'^2\not=0$ imply that  there are $m,c\in\r$ such that $b(u)=m a(u)+c$, where $m^2=1$ or $m^2=7\pm 4\sqrt{3}$. We see that this gives a contradiction. Coming back to the expressions of $A_3$ and $B_3$, we hae
\begin{equation*}
\begin{split}
A_3&=\frac{1}{4} \alpha  r^3 a'^2 \left(\left(3 m^2-1\right) u a'+\left(1-3 m^2\right) a-2 c m\right)\\
B_3&=\frac{1}{4} \alpha  r^3 a'^2 \left(\left(m^2-3\right) m u a'-\left(m^2-3\right) m a-c m^2+c\right).
\end{split}
\end{equation*}
From equations $A_3=B_3=0$, a first case to consider is $a'=0$ identically. Since $b'=ma'$, we have  $b'=0$ identically. Then the functions $a$ and $b$ are constant, hence the surface is rotational. This proves the result. 

Suppose now $a'\not=0$ and we will see that this case is not possible. A combination of the   equations $A_3=0$ and $B_3=0$ implies $c=0$ and $a-ua'=0$. Solving this equation, we have $a(u)=c_1 u$, $c_1\not=0$ and, consequently, $b(u)=m c_1 u$ as well. The computation of the coefficient $B_2$   now gives
$$B_2=c_1^2m^2r^3(\alpha r+u(4-\alpha)r').$$
Then $B_2=0$ implies $\alpha r+u(4-\alpha)r'=0$. Notice that if $4-\alpha =0$, then $r=0$, which it is not possible. Thus, $4-\alpha\not=0$ and the solution of the differential equation $\alpha r+u(4-\alpha)r'=0$ is 
$$r(u)=c_2u^{\frac{\alpha}{\alpha-4}},\quad c_2\not=0.$$
Finally, the computation of the equation $A_1=0$ yields
$$(\alpha -4) (\alpha -2)  \left(c_1^2 \left(m^2+1\right)+1\right)u^2-\alpha  (\alpha +4) c_2^2 u^{\frac{2 \alpha }{\alpha -4}}=0.$$
Since the functions $\{u^2,u^{\frac{2 \alpha }{\alpha -4}}\}$ are linearly independent, their coefficients must vanish. This implies   $(\alpha-4)(\alpha-2)=0$ and $\alpha(\alpha+4)=0$. Since this is not possible, we arrive to the desired contradiction, proving the result.

\section{Proof of Theorem \ref{t3}: case $\alpha=-4$}\label{s5}

We prove Thm. \ref{t3} when  $\alpha=-4$. This case is particular because the family of $(-4)$-stationary surfaces are related with the family of minimal surfaces, which we now explain. Consider the inversion   from $0$ given by 
$$\Phi\colon\r^3-\{0\}\to\r^{3}-\{0\},\quad \Phi(p)=\frac{p}{|p|^2}.$$
The following result was proved in \cite{lo1}
\begin{theorem}\label{tlo}
Let $\Sigma$ be a surface in Euclidean space. If $\Sigma$ is an $\alpha$-stationary surface, then $\Phi(\Sigma)$ is an $-(\alpha+4)$-stationary surface.
\end{theorem}
 According to this theorem, the value $\alpha=-2$ is special in the sense that the mapping $\alpha\mapsto -(\alpha+4)$ is the symmetry about the value $-2$. Thus the inversion of a $(-2)$-stationary surface is again a $(-2)$-stationary surface. The case $\alpha=-4$ is interesting because Thm. \ref{tlo} says that the inversion  of a $(-4)$-stationary surface is a minimal surface ($H=0$) and vice versa.  

On the other hand, the inversion $\Phi$   maps the class of curves formed by straight-lines and circles into itself. Some interesting examples are the following.
\begin{enumerate}
\item A circle crossing $0$ is mapped via $\Phi$ to a straight-line which does not contain $0$.
\item A circle that does not cross $0$ is mapped via $\Phi$ to another circle which does not contains $0$.
\item Any straight-line crossing $0$ remains invariant via $\Phi$.
\item Any straight-line that does not cross $0$ is mapped via $\Phi$ to a circle which contains $0$.
\item As a consequence, the inversion of a rotational surface whose rotation axis crosses $0$ is a rotational surface with the same axis.
\end{enumerate}
Therefore the image of a cyclic surface  by the inversion $\Phi$ is a surface foliated by circles and straight-lines. Similarly, the image of a ruled surface by the inversion $\Phi$ is a surface foliated by circles and straight-lines.

The classification of cyclic $(-4)$-stationary surfaces is now clear thanks to Thm. \ref{tlo}. Let $\Sigma$ be a cyclic $(-4)$-stationary surface. Then $\Phi(\Sigma)$ is a minimal surface    foliated by circles and straight-lines.  It is known that the only minimal surfaces foliated by   circles and straight-lines are  planes,   helicoids,   catenoids or      the Riemann minimal examples. In consequence, the cyclic $(-4)$-stationary surfaces are the image by the inversion $\Phi$ of these minimal surfaces.   See  Fig. \ref{fig1}. This proves the item (2) of Thm. \ref{t3}. Notice that the inversion of a plane is a vector plane (if the initial plane contains $0$) or a sphere crossing $0$, otherwise.  

  \begin{figure}[hbtp]
\begin{center}
\includegraphics[width=.3\textwidth]{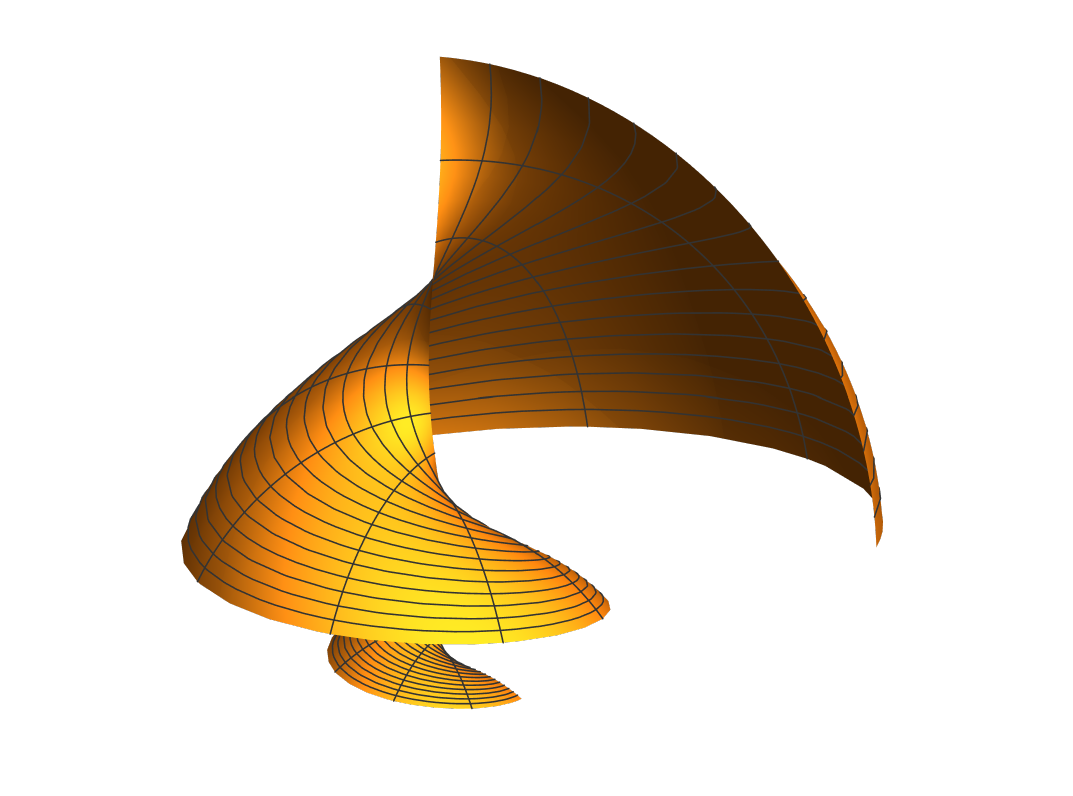},\quad \includegraphics[width=.25\textwidth]{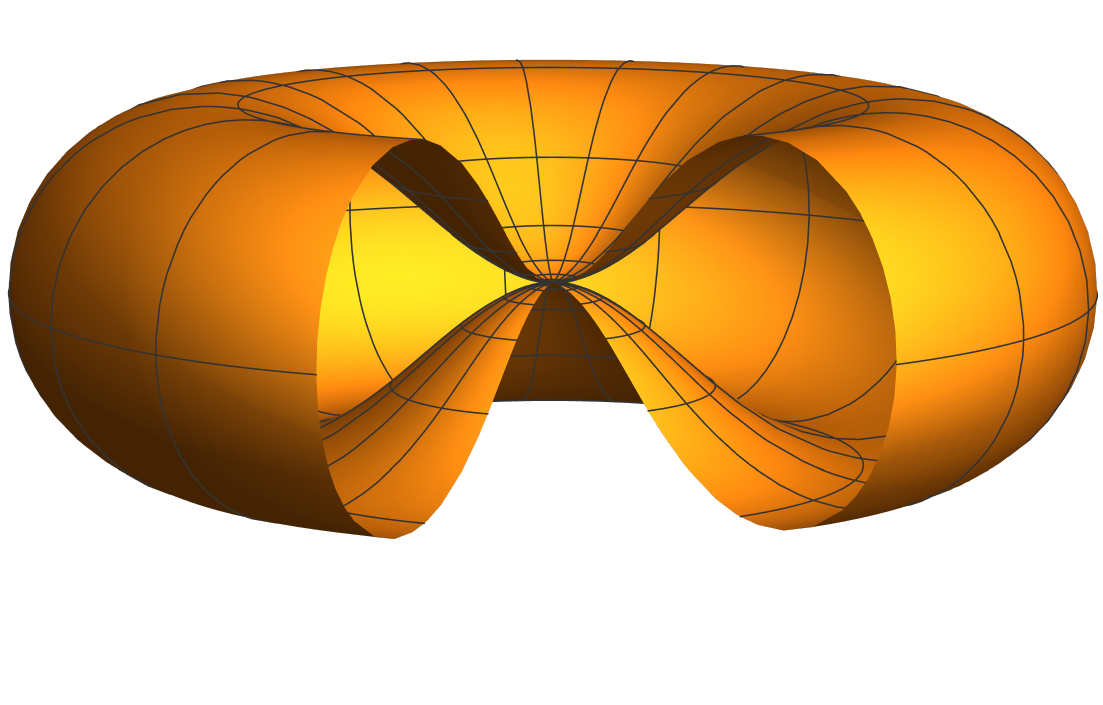}\quad \includegraphics[width=.2\textwidth]{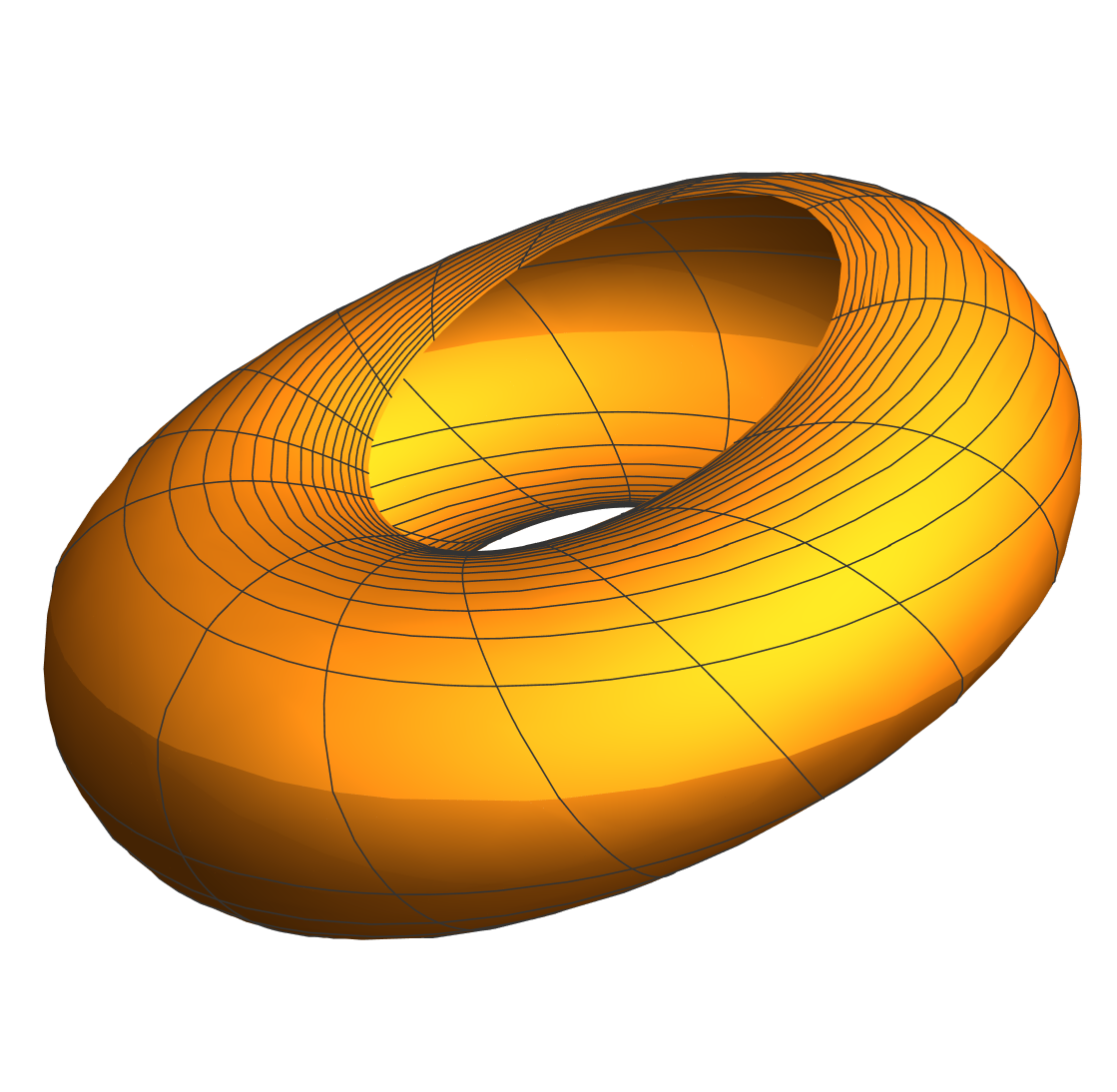}
\end{center}
\caption{Examples of $(-4)$-stationary surfaces obtained by inversions of a helicoid (left), a catenoid (middle) and a Riemann minimal example (right).}\label{fig1}
\end{figure}
 
\section{Proof of Theorem \ref{t3}: case $\alpha\not=-4,-2$}\label{s6}

In this section, and for $\alpha\not=-4,-2$, we study $\alpha$-stationary surfaces $\Sigma$ foliated by circles where the planes of the foliation are not parallel. First, we need  a suitable parametrization of $\Sigma$. The proof is based on  \cite{ni} where Nitsche proved that the only non-zero cyclic constant mean curvature surfaces are rotational surfaces or spheres. Let $\Psi=\Psi(u,v)$ be a parametrization of $\Sigma$ where the parametric curves $v\mapsto \Psi(u,v)$ are   circles. For each $u$, let  $\Pi_u\subset\r^3$ be the affine plane containing this circle.   Let $Z$ be a smooth unit vector field perpendicular to each plane $\Pi_u$ and let  $\Gamma=\Gamma(u)$ be an integral curve of $Z$ parametrized by arc-length. Let 
$$\T(u):=\Gamma'(u)=Z(\Gamma(u)),$$
 where $\T$ is the unit tangent vector to
$\Gamma$. Since the planes of the foliation of $\Sigma$ are not parallel, the curve $\Gamma$ is not a straight-line, in particular, $\Gamma$ has associated  a well-defined Frenet frame, which is denoted by  $\{\T,\N,\B\}$, where $\N$ and $\B$ are the normal and binormal vectors, respectively. Definitively, a parametrization of $\Sigma$ is given by  
\begin{equation}\label{parame}
\Psi(u,v)=\C(u)+r(u)(\cos{v}\ \N (u)+\sin{v}\ \B (u)),
\end{equation}
where $r=r(u)>0$ and $\C=\C(u)$ denote the radius and the center of each $u$-circle $\Sigma\cap\Pi_u$, respectively. The Frenet equations of $\Gamma$ are
\begin{eqnarray*}
\T'&=&\hspace*{.8cm}\kappa \N\\
\N'&=&-\kappa \T+\ \ \tau \B\\
\B'&=&\hspace*{.6cm}-\tau \N,
\end{eqnarray*}
where the prime $'$ denotes the derivative with respect to the
$u$-parameter and $\kappa$ and $\tau$ are the curvature and torsion
of $\Gamma$, respectively. We will need to express $\C$ in coordinates with respect to the Frenet frame, namely, 
\begin{equation}\label{alfa}
\C=a \T+b\N+c\B,
\end{equation}
where $a,b,c$  are smooth functions on the variable $u$.

Once we have the parametrization $\Psi$, we compute Eq. \eqref{eq1}. After a lengthy computations using the symbolic software Mathematica, this equation  becomes 
\begin{equation}\label{formula}
 \sum_{n=0}^4 (A_n(u) \cos{(nv)} +B_n(u)\sin{(n v)})=0.
\end{equation}
Therefore,  all coefficients $A_n$ and $B_n$ vanish identically in the domain of the  variable $u$.   
For $n=4$, we have
\begin{equation*}
\begin{split}
A_4&=\frac{1}{8} (\alpha +4) r^4 \kappa  \Big(2 c' (c (a \kappa +b'-c \tau )+b^2 \tau )\\
&-b \left(2 a \kappa  (b'-2 c \tau )+a^2 \kappa ^2-4 c \tau  b'+b'^2-\tau ^2 (b^2-3 c^2)+r^2 \kappa ^2\right)+b c'^2\Big),\\
B_4&=-\frac{1}{8} (\alpha +4) r^4 \kappa  \Big(2 a \kappa  \left(c \left(b'-c \tau \right)+b c'+b^2 \tau \right)+a^2 c \kappa ^2\\
&+c \left(b'^2-\left(b \tau +c'\right) (3 b \tau +c')+r^2 \kappa ^2\right)+2 b b' (b \tau +c')-2 c^2 \tau  b'+c^3 \tau ^2\Big).
\end{split}
\end{equation*}
The linear combination $cA_4-b B_4=0$ gives 
\begin{equation}\label{cases}
(4+\alpha)r^4\kappa  \left(b^2+c^2\right) \left(b \tau +c'\right) \left(a \kappa +b'-c \tau \right)=0.
\end{equation}
Since $\alpha+4\not=0$, we have three cases to discuss.
\begin{enumerate}
\item Case $b^2+c^2=0$ identically. The equation $A_3=0$ is 
$$  (\alpha +2) r^3 \kappa ^3 \left(a^2+r^2\right) =0.$$
This gives a contradiction because $r\not=0$ and $2+\alpha\not=0$.  
 
\item Case $b\tau+c'=0$. Then $c'=-b\tau$. The case $b^2+c^2=0$ will be discarded. Equations $A_4=0$ and $B_4=0$ become
\begin{equation*} 
\begin{split}
0&=(\alpha +4) b    \left(2 a \kappa  \left(b'-c \tau \right)+a^2 \kappa ^2+\left(b'-c \tau \right)^2+r^2 \kappa ^2\right)\\
0&=(\alpha +4) c  \left(2 a \kappa  \left(b'-c \tau \right)+a^2 \kappa ^2+\left(b'-c \tau \right)^2+r^2 \kappa ^2\right).
\end{split}
\end{equation*}
Since $b\not=0$ and $c\not=0$, the parenthesis must vanish. However, this equation can be viewed as  a polynomial equation of degree $2$ on $a$ whose discriminant is $-4r^2\kappa^2<0$. This gives a contradiction. 

\item Case $a \kappa +b'-c \tau =0$. Then $b'=c\tau-a\kappa$. We also suppose $b^2+c^2\not=0$ and $b\tau+c'\not=0$ by the two previous cases. Then

\begin{equation*} 
\begin{split}
A_4&=\frac{1}{8} (\alpha +4) b r^4 \kappa  \left(\left(b \tau +c'\right)^2-r^2 \kappa ^2\right)\\
B_4&=\frac{1}{8} (\alpha +4) c r^4 \kappa  \left(\left(b \tau +c'\right)^2-r^2 \kappa ^2\right).
\end{split}
\end{equation*}
  This implies $b\tau+c'=\pm r\kappa$. Assume $b\tau+c'=  r\kappa$. Using this identity together $a \kappa +b'-c \tau =0$, we can get an expression of $b''$ in terms of $c$ and $c'$. By substituting into $A_3=0$ and $B_3=0$, we obtain 
\begin{equation*} 
\begin{split}
A_3&=\frac{1}{2} (\alpha+5)r^5 \kappa^2  (b^2\kappa-a'b+cr'),\\
B_3&=-\frac{1}{2} (\alpha+5) r^5 \kappa^2  (a'c-bc\kappa+br').
\end{split}
\end{equation*}
The linear combinations $bA_3+cB_3=0$ and $cA_3-bB_3=0$ give, respectively, 
\begin{equation*} 
\begin{split}
0&=(\alpha +5) \kappa  \left(b^2+c^2\right) \left(b \kappa -a'\right),\\
0&=(\alpha+5)r'\kappa=0.
\end{split}
\end{equation*}
In case that   $b\tau+c'=- r\kappa$, we obtain the same two equations.  We distinguish two cases. 
\begin{enumerate}
\item Case $\alpha\not=-5$. We deduce from the first equation that $a'=b\kappa$ and  from the second equation, that $r$ is a constant function $r(u)=r_0\not=0$. Now the computations of equations  $A_2=0$ and $B_2=0$ imply 
\begin{equation*} 
\begin{split}
0&= r_0^6 c \kappa ^3\\
0&= r_0^6 b \kappa ^3.
\end{split}
\end{equation*}
  This is a contradiction because $b,c\not=0$.  
 \item Case $\alpha=-5$. Consider the surface $\Phi(\Sigma)$. This surface is a $1$-stationary surface (Thm. \ref{tlo}) which it is a cyclic surface or a ruled surface. Suppose that $\Phi(\Sigma)$ is a cyclic surface. By the last   case (a), the planes of the foliation must be parallel. Now Thm. \ref{t2} implies that the surface is rotational. By \cite[Prop. 4.1]{dl2}, the axis crosses $0$ or it is a sphere containing $0$. The latter case is not possible by the value of $\alpha$. Thus the axis crosses $0$. Coming back to $\Sigma$ via $\Phi$, the surface $\Sigma$ is a rotational surface with the same rotation axis, which is contradictory because it was initially assumed that the planes of the foliation were are not parallel. 
 
 Definitively, $\Phi(\Sigma)$ is a ruled surface. By the classification given in Thm. \ref{t1},   $\Phi(\Sigma)$ is the ruled surface parametrized by \eqref{heli}. This proves the item (1) of Thm. \ref{t3}, completing the proof.  
 \end{enumerate}
  \end{enumerate}
\section{The case $\alpha=-2$: cyclic $(-2)$-stationary surfaces}\label{s7}

In this section, we assume $\alpha=-2$. Recall that the value $\alpha=-2$ is special in Thm. \ref{tlo}. An example of $(-2)$-stationary surface is any sphere centered at $0$. This sphere can be parametrized by a cyclic surface without to be a parametrization of a rotational surface. Let $\Sigma$ be a $(-2)$-stationary surface such that   the planes of the foliation are not parallel. Here we follow the same notation as in  Sect. \ref{s5}. Notice that  $|\Psi(u,v)|^2=r^2+a^2$. By observing   the discussion on cases in \eqref{cases}, we find that $b\tau+c'=0$ and $a\kappa+b'-c\tau=0$ are not possible. 

It remains to discuss the case  $b^2+c^2=0$. Thus $b=c=0$. In consequence, by \eqref{alfa}, we have 
$\C=a\T$. In such a case, the coefficients $A_3$, $B_3$, $A_2$ and $B_2$ are trivially $0$.   Now $B_1=0$ is 
$$(a^2+\kappa^2)\tau(aa'+rr')=0.$$
 Since $\kappa\not=0$, then $aa'+rr'=0$ or $\tau=0$. 
 \begin{enumerate}
 \item Case  $aa'+rr'=0$ identically. Then the function $a^2+r^2$ is constant. Since $|\Psi(u,v)|^2=r^2+a^2$, we conclude that  $\Sigma$ is a sphere centered at $0$.
 
 \item Case $\tau=0$. Then Eq. \eqref{formula} reduces to $A_1=0$ and $A_0=0$. Both equations describe all cyclic $(-2)$-stationary surfaces, namely, 
 \begin{align} 
0=&
 a^2 r \left(\kappa  \left(-3 a'^2+r r''+3 r'^2\right)-r r' \kappa '\right)+r^3 \left(\kappa  \left(a'^2+r r''-r'^2\right)-r r' \kappa '\right)\nonumber\\
 &+a^3 \left(r \kappa  a''+a' \left(2 \kappa  r'-r \kappa '\right)\right)+a r^2 \left(r \kappa  a''-a'  (6 \kappa  r'+r \kappa')\right),\label{21}\\
 0=&a r r' \left(2 \left(a'^2+r'^2\right)+r^2 \kappa ^2\right)+a^4 \kappa ^2 a'+a^2 \left(r a'' r'-r a' r''+a' r'^2+r^2 \kappa ^2 a'+a'^3\right)\nonumber\\
 &-r^2 \left(-r a'' r'+a' \left(r r''+r'^2\right)+a'^3\right)+a^3 r \kappa ^2 r'.\label{22}
 \end{align}
 \end{enumerate}

Definitively, we have proved the following result.

\begin{theorem} Let $\Sigma$ be a non-spherical cyclic  $(-2)$-stationary surface whose planes of the foliations are not parallel. Suppose that the surface is parametrized by \eqref{parame}. Then the curve   $\Gamma=\Gamma(u)$ is planar, where   $\C=a \T$. Moreover, the functions $r(u)$, $a(u)$ and $\kappa(u)$ are characterized by \eqref{21} and \eqref{22}.
\end{theorem}

We give an example of these surfaces. For this, let $a=0$ identically. Then \eqref{22} is trivially $0$ and \eqref{21} is 
$$\kappa  \left(r r''-r'^2\right)-r r' \kappa '=0,$$
or equivalently
$$\frac{rr''-r'^2}{rr'}=\frac{\kappa'}{\kappa}.$$
By solving, we deduce  that there is $m\in\r$, $m\not=0$, such that 
\begin{equation}\label{r3}
\frac{r'}{r}=m\kappa.
\end{equation}
Thus
$$r(u)=e^{m\int \kappa}.$$
In order to show an example, take $m=1$ and  
$$ r(u)=\frac{1}{\kappa(u)}.$$
From \eqref{r3}, we deduce 
$$r(u)=u,\quad \kappa(u)=\frac{1}{u},\quad u>0.$$
We now find the   parametrization \eqref{parame} of $\Gamma$. Since  $\Gamma$ is a planar curve,  after a linear isometry of $\r^3$, we can assume that $\Gamma$ is contained in a horizontal plane. Then the binormal vector is  $\B(u)=(0,0,1)$ and 
$$\Gamma'(u)=\T(u)=(\cos\theta(u),\sin\theta(u),0),$$
where 
$$\theta'(u)=\kappa(u)=\frac{1}{u}.$$
This gives $\theta(u)=\log(u)$, hence 
$$\N(u)=(-\sin(\log(u)),\cos(\log(u)),0).$$
Finally, the parametrization \eqref{parame} is 
\begin{equation}\label{new}
\begin{split}
\Psi(u,v)&=r(u)(\cos v\N(u)+\sin v\B(u)\\
&=\left(-u\sin(\log(u))\cos v,u\cos(\log(u))\cos v,u\sin v\right).
\end{split}
\end{equation}
A picture of this surface is given Fig. \ref{fig2}.
 \begin{figure}[hbtp]
\begin{center}
\includegraphics[width=.4\textwidth]{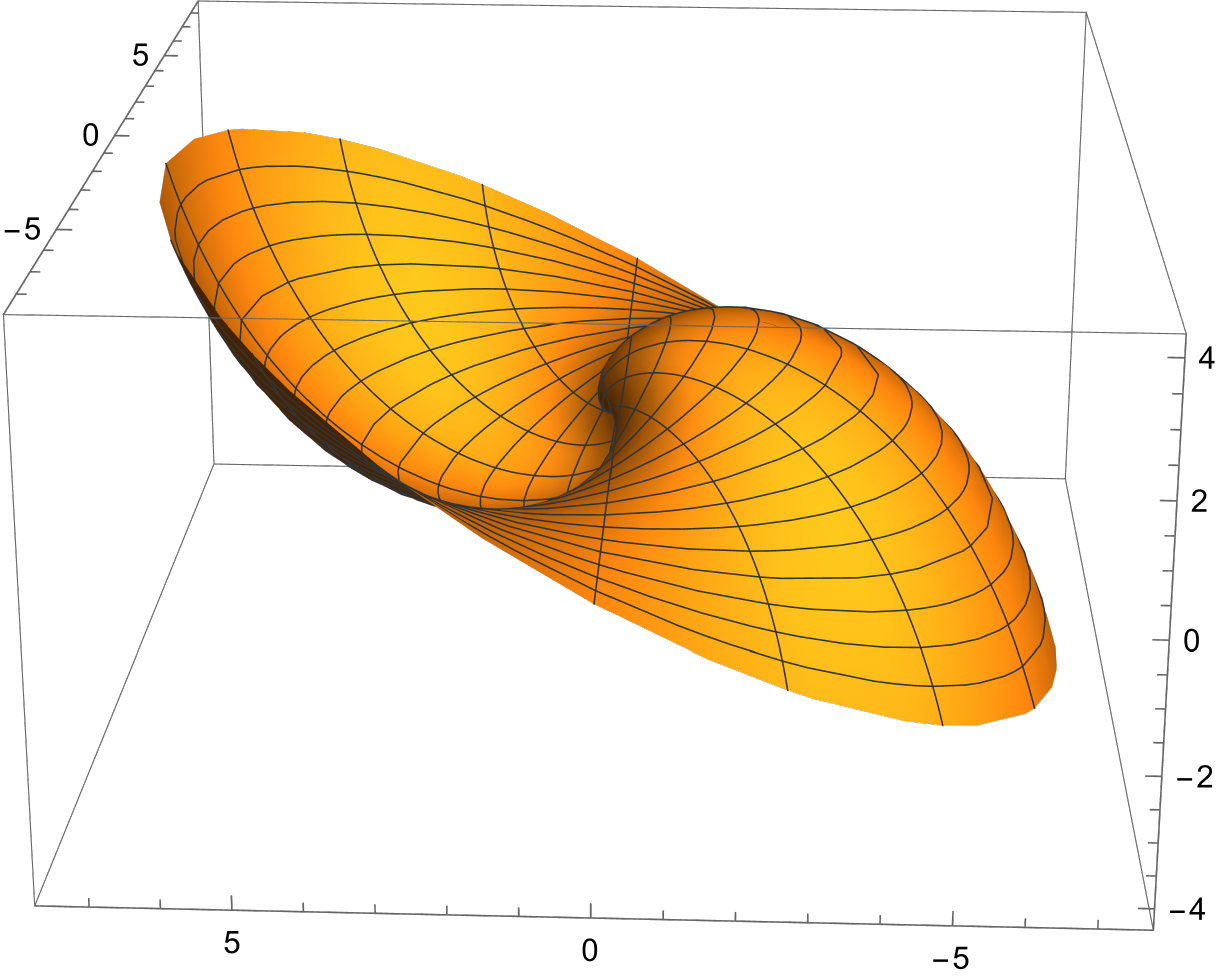},\quad \includegraphics[width=.4\textwidth]{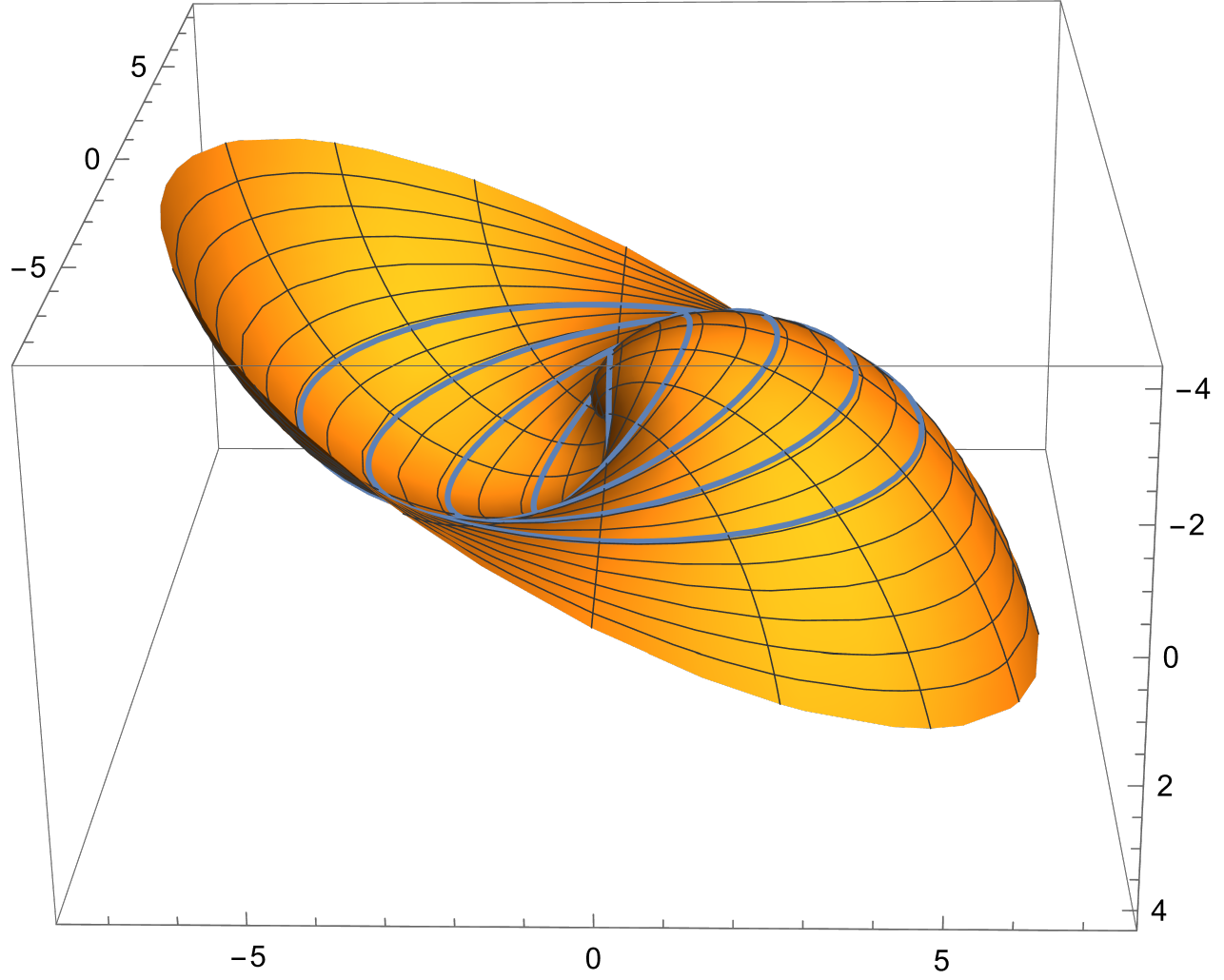}
\end{center}
\caption{ An example of a non-spherical $(-2)$-stationary surface foliated by circles in non-parallel planes. Left: surface parametrized by \eqref{new}. Right: some circles of the foliation.}\label{fig2}
\end{figure}
\section*{Acknowledgements}
The author   would like to thank the anonymous referee for their careful reading of the paper. Specially his/her contribution in improve    Thm. \ref{t1} with  the examples of ruled $\alpha$-stationary surfaces for $\alpha=1$.

The author has been partially supported by MINECO/MICINN/FEDER grant no. PID2023-150727NB-I00,  and by the ``Mar\'{\i}a de Maeztu'' Excellence Unit IMAG, reference CEX2020-001105- M, funded by MCINN/AEI/10.13039/ 501100011033/ CEX2020-001105-M.

 \section*{Affiliations}

\address{Department of Geometry and Topology\\ University of Granada. 18071 Granada, Spain.}\\
\email{rcamino@ugr.es}\\
\textbf{ORCID ID: 0000-0003-3108-7009 }


\begin{thebibliography}{99}
 

\bibitem{al} Abresch, U., Langer, J.:  \emph{The normalized curve shortening flow and homothetic solutions}. J. Differ. Geom. \textbf{23}, 175--196 (1986).

\bibitem{an} Anciaux, H.: \emph{Two non existence results for the self-similar equation in Euclidean 3-space}. J. Geom. \textbf{96}, 1--10 (2009).
\bibitem{ay}      Aydin,  M. E.,   Erdur Kara, A.:  \emph{Ruled singular minimal surfaces}. J. Geom. Phys. \textbf{195}, Paper No. 105055, 8 pp. (2024).

 \bibitem{ca}   Carath\'edory, C.:  \emph{Variationsrechnung und Partielle Differentialgleichungen erster Ordnung}. Teubner, Leipzig und Berlin (1935).


 
 \bibitem{d4}   Dierkes, U.: \emph{Minimal cones and a problem of Euler}.  Rend. Sem. Mat. Univ. Padova (2024), published online first.




\bibitem{dh}  Dierkes,  U.,   Huisken, G.:
\emph{The $n$-dimensional analogue of a variational problem of Euler}. Math. Ann. \textbf{389}, 3841--3863 (2024).


\bibitem{dl1} Dierkes, U., L\'opez, R.:  \emph{On the Plateau problem for surfaces with minimum moment of inertia}. To appear in Adv. Calc. Var. 

\bibitem{dl2} Dierkes, U., L\'opez, R.: \emph{Axisymmetric stationary surfaces for the moment of inertia}. arXiv:2507.12392 [math.DG] (2025).


\bibitem{eu}  Euler, L.: \emph{Methodus Inveniendi lineas curvas maximi minimive proprietate gaudentes sive solutio problematis isoperimetrici latissimo sensu accepti}, Lausanne et Genevae. (1744).

\bibitem{hi}  Hieu, D. T.,   Hoang, N. M.: \emph{Ruled minimal surfaces in with density $e^z$}. Pacific J. Math. \textbf{243},   277--285   (2009).
\bibitem{lo01} L\'opez, R.: \emph{Invariant surfaces in Euclidean space with a log-linear density}.  Adv. Math. \textbf{339}, 285--309 (2018) .
\bibitem{lo02}  L\'opez, R.: \emph{The one dimensional case of the singular minimal surfaces with density}. Geom. Dedic. \textbf{200}, 302--320  (2019).

\bibitem{lo03} L\'opez, R.: \emph{Ruled surfaces of generalized self-similar solutions of the mean curvature flow}. Mediterr. J. Math. \textbf{18}, Paper No. 197, 12 pp (2021).

\bibitem{lo1} L\'opez, R.: \emph{A connection between minimal surfaces and the two-dimensional analogues of a problem of Euler}.  Ann. Mat. Pura Appl. (2025). https://doi.org/10.1007/s10231-025-01593-w.

\bibitem{lo2}   L\'opez, R.: \emph{Stationary surfaces for   the moment of  inertia with constant Gauss curvature}.  Analysis (2025). https://doi.org/10.1515/anly-2025-0037.

\bibitem{ma}   Mason, M.: \emph{Curves of minimum moment of inertia with respect to a point}. Annals of Math. \textbf{7 }, 165--172 (1906). 

\bibitem{ni}   Nitsche, J. C. C.: \emph{Cyclic surfaces of constant mean curvature}.  Nachr. Akad. Wiss. Gottingen Math. Phys. II {\bf 1}, 1--5  (1989).



\bibitem{to}   Tonelli, L.: \emph{Fondamenti di calcolo della variazioni}. \textbf{1} (1921).


\end{thebibliography}
 \end{document}